\documentclass[11pt, reqno]{amsart}
\usepackage{amsmath,amssymb,amsthm,amsfonts,verbatim}
\usepackage{microtype}
\usepackage[all,2cell]{xy}
\usepackage{mathtools}
\usepackage{graphicx}
\usepackage{pinlabel}
\usepackage{hyperref}
\usepackage{mathrsfs}
\usepackage{color}
\usepackage[dvipsnames]{xcolor}
\usepackage{enumitem}
\usepackage{cite}
\usepackage{soul}

\allowdisplaybreaks

\CompileMatrices

\usepackage[top=1.2in,bottom=1.2in,left=1in,right=1in]{geometry}

\usepackage{hyperref} 
\hypersetup{          
	colorlinks=true, breaklinks, linkcolor=[RGB]{51 102 204}, filecolor=Orchid, urlcolor=[RGB]{51 102 204},
	citecolor=Orchid, linktoc=all, }
\usepackage[nameinlink]{cleveref}

\theoremstyle{plain}
\newtheorem{theorem}{Theorem}[section]
\newtheorem{proposition}[theorem]{Proposition}
\newtheorem{lemma}[theorem]{Lemma}

\newtheorem{corollary}[theorem]{Corollary}

\theoremstyle{definition}

\newtheorem{remark}[theorem]{Remark}

\newcommand{\ZZ}{\mathbb{Z}}
\newcommand{\QQ}{\mathbb{Q}}

\DeclareMathOperator{\Diff}{Diff}
\DeclareMathOperator{\Sp}{Sp}

\DeclareMathOperator{\Mod}{Mod}

\DeclareMathOperator{\Aut}{Aut}
\DeclareMathOperator{\Out}{Out}

\DeclareMathOperator{\Sym}{Sym}

\counterwithin{figure}{section}
\counterwithin{equation}{section}

\title{Abelianization of Symmetric Mapping Class Groups}
\author{Xiyan Zhong}

\begin{document}
\maketitle
\begin{abstract}
Let $\widetilde{S}\to S$ be an unbranched regular $p$-fold cyclic cover of a closed orientable surface $S$ of genus $g$. Two natural groups are associated to this cover. The first is the centralizer $\Mod(\widetilde{S},\sigma)$ of a chosen generator $\sigma$ of the deck transformation group in $\Mod(\widetilde{S})$. The second is the finite-index subgroup $\Mod(S,[\beta])$ of $\Mod(S)$ consisting of mapping classes that fix the nonzero class $[\beta]\in H_1(S;\mathbb Z/p\mathbb Z)$ corresponding to the cover. For $p=2$, Sato \cite{sato2} computed the abelianizations of these groups. We compute their abelianizations for every odd prime $p$ and show that they exhibit a splitting phenomenon different from the case $p=2$. We also construct explicit abelianization maps using Morita's crossed homomorphism and the Prym representation.
\end{abstract}

\section{Introduction}
Let $S=S_{g,r}^n$ be an orientable surface of genus $g\ge 3$ with $r$ punctures and $n$ boundary components. We omit $r$ or $n$ from the notation whenever it is $0$. The mapping class group of $S$ is \[\Mod(S)=\pi_0(\Diff^+_{\partial}(S))\] 
where $\Diff^+_{\partial}(S)$ denotes the group of orientation-preserving homeomorphisms of $S$ that fixes the boundary pointwise. 

Finite-index subgroups of the mapping class group play an important role in low-dimensional topology and geometric group theory, but many of their basic properties remain not fully understood. Two major open problems concern finite-index subgroups of $\Mod(S)$. The first \cite[Conjecture~1]{Ivanov} is whether $\Mod(S)$ has the congruence subgroup property, i.e.\ whether every finite-index subgroup contains a congruence subgroup
\[\text{Ker}(\Mod(S)\to \Out(\pi_1(S)/N)\] 
where $N\vartriangleleft \pi_1(S)$ is a finite-index characteristic subgroup. The second is the well-known Ivanov conjecture \cite[Question~7]{Ivanov}, asserting that every finite-index subgroup has finite abelianization. Although this remains open even for congruence subgroups, it is known for finite-index subgroups containing the Torelli group by McCarthy \cite{McCarthy} (for $S$ closed) and Hain \cite{Hain}, and more generally for finite-index subgroups containing sufficiently deep terms of the Johnson filtration by Ershov-He \cite{ErshovHe}.


From now on, we assume that $S$ has at most one puncture or one boundary component. 

In this paper, we study the finite-index subgroup
\[\Mod(S,[\beta])\coloneq \text{Stab}_{\Mod(S)}([\beta]),\quad [\beta]\in H_1(S;\ZZ/p\ZZ)^*,\]
where $p$ is a prime. This subgroup arises naturally from cyclic covering spaces. The nonzero homology class $[\beta]$ determines an unbranched cyclic $p$-fold cover
\[\widetilde{S}\to S\]
via the homomorphism
\[\pi_1(S)\to \ZZ/p\ZZ,\quad
\gamma\mapsto \widehat{i}(\gamma,\beta)\bmod p,\]
where $\beta$ is a simple closed curve representing $[\beta]$, and $\widehat{i}$ denotes the algebraic intersection number. By covering space theory, $\Mod(S,[\beta])$ is precisely the subgroup of $\Mod(S)$ consisting of mapping classes that admit lifts to $\widetilde{S}$ which commute with the deck transformation group. The induced action on the homology of the cover gives rise to the Prym representation
\[\Mod(S,[\beta])\longrightarrow \Aut(H_1(\widetilde{S};\QQ))/\text{deck transformations}\]
whose image was determined by Looijenga \cite{PrymRep}. On the algebraic-geometric side, this representation is closely related to Prym varieties \cite{PrymVariety}.

Since $\Mod(S,[\beta])$ contains the Torelli subgroup, its abelianization is finite by the results above. A natural question is therefore to determine this finite abelian group explicitly. For $p=2$, Sato \cite[Theorem 0.2]{sato2} determined the abelianization of $\Mod(S,[\beta])$ for $g \ge 4$. Our first main theorem determines the abelianization of $\Mod(S,[\beta])$ for every odd prime $p$.
\begin{theorem}\label{thm: Abel Mod beta}
Let $g\ge 4$ and $p$ be an odd prime number. Then
\[H_1(\Mod(S_g,[\beta]);\ZZ)\cong\begin{cases}
         \ZZ/p\ZZ & \text{ if } g \not\equiv 1 \pmod p; \\
        \ZZ/p\ZZ \oplus \ZZ/p\ZZ  & \text{ if } g \equiv 1 \pmod p.
    \end{cases}\]
If $S=S_g^1$ or $S_{g,1}$, then
\[H_1(\Mod(S,[\beta]);\ZZ)\cong \ZZ/p\ZZ \oplus \ZZ/p\ZZ.\]
\end{theorem}
The main difference between our result for odd primes and Sato's result for $p=2$ is that every occurrence of $\ZZ/p\ZZ\oplus\ZZ/p\ZZ$ in our theorem is replaced by $\ZZ/4\ZZ$ in the case $p=2$. To establish the splitting in the $\ZZ/p\ZZ\oplus\ZZ/p\ZZ$ cases, we explicitly construct a surjection (see \eqref{eq: abel 1})
\[\Mod(S,[\beta])\longrightarrow \ZZ/p\ZZ\oplus\ZZ/p\ZZ\] using Morita's crossed homomorphism \cite[Section 6]{MoritaI} from $\Mod(S_{g,1})$ to $H_1(S_g;\ZZ)$.

The distinction between the abelianizations of $\Mod(S_g,[\beta])$ in the cases $g\not\equiv1\pmod p$ and $g\equiv1\pmod p$ is determined by whether the image of the Torelli subgroup in $H_1(\Mod(S_g,[\beta]);\ZZ)$ is trivial or nontrivial. Moreover, the explicit abelianization map constructed in the proof of Theorem~\ref{thm: Abel Mod beta} gives the following coset representatives for the commutator subgroup of $\Mod(S,[\beta])$.
\begin{corollary}\label{cor: coset beta}
Let $\beta$ be a simple closed curve in $S$, and let $\alpha$ and $\alpha'$ be non-isotopic, disjoint, homologous simple closed curves that each intersect $\beta$ once and together separate $S$. Then the following elements form a complete set of coset representatives for the commutator subgroup of $\Mod(S,[\beta])$:
\[
(T_\alpha^p)^i(T_\alpha T_{\alpha'}^{-1})^j,\qquad 0\le i,j\le p-1,
\]
except in the case where $S=S_g$ and $g\not\equiv1\pmod p$, in which case the representatives are
\[
{(T_\alpha^p)^i,\quad 0\le i\le p-1}.
\]
\end{corollary}

Although mapping classes in $\Mod(S,[\beta])$ lifts to $\widetilde{S}$, the lift is not canonical, differing by deck transformations. To obtain a canonical action on the homology of the cover, it is therefore natural to consider the subgroup of $\Mod(\widetilde{S})$ that commute with the deck transformations:
\[\Mod(\widetilde{S},\sigma)\coloneq \text{C}_{\Mod(\widetilde{S})}(\sigma),\quad \sigma \text{ generates the deck transformation group}.\]
This group is called the symmetric mapping class group and it admits the Prym representation
\begin{equation}\label{map: prym rep}
  \text{Prym}: \Mod(\widetilde{S},\sigma) \longrightarrow \Aut(H_1(\widetilde{S};\QQ)).
\end{equation}
When $S=S_g$, Birman-Hilden \cite{BirmanHilden} identifies $\Mod(\widetilde{S},\sigma)=\pi_0(C_{\Diff^+(\widetilde{S})}(\sigma))$, which yields the central short exact sequence
\[1\to \langle\sigma\rangle \to \Mod(\widetilde{S},\sigma) \to \Mod(S,[\beta])\to 1.\]
When $S=S_{g,1}$ or $S=S_g^1$, the deck transformation does not define an element of $\Mod(\widetilde{S})$, since mapping classes are required to fix punctures or boundary components pointwise. In this case, lifting gives an isomorphism
\[\Mod(\widetilde{S},\sigma)\cong \Mod(S,[\beta]).\]

For $p=2$, Sato \cite[Theorem 0.2]{sato2} showed that the abelianization of $\Mod(\widetilde{S},\sigma)$ is $\ZZ/4\ZZ$, using the Schottky theta constant. We determine the abelianization of $\Mod(\widetilde{S},\sigma)$ for every odd prime $p$.
\begin{theorem}\label{thm: Abel Mod Sigma}
Let $g\ge 4$ and $p$ be an odd prime number. If $S=S_g$, $S_g^1$, or $S_{g,1}$ then
\[H_1(\Mod(\widetilde{S},\sigma);\ZZ)\cong \ZZ/p\ZZ\oplus \ZZ/p\ZZ.\]
\end{theorem}
When $g\not\equiv1\pmod p$, this result has a geometric interpretation via the Prym representation. In this case, the Prym representation \eqref{map: prym rep} induces a surjection
\[H_1(\Mod(\widetilde{S},\sigma);\ZZ) \to H_1(\text{Image}(\text{Prym}))\to \ZZ/p\ZZ\oplus \ZZ/p\ZZ,\]
for odd $p$, whereas this surjectivity fails when $p=2$.

Moreover, the explicit abelianization map constructed in the proof of Theorem~\ref{thm: Abel Mod Sigma} provides the following coset representatives for the commutator subgroup of $\Mod(\widetilde{S},\sigma)$.
\begin{corollary}\label{cor: coset sigma}
   Let $\beta$ be a simple closed curve in $S$, and let $\widetilde{S}\to S$ be the $p$-fold cyclic cover determined by the mod $p$ intersection number with $\beta$, with deck transformation group $\langle\sigma\rangle$. Let $\alpha$ and $\alpha'$ be nonisotopic, disjoint, homologous simple closed curves that each intersect $\beta$ once and together separate $S$. Let $\widetilde{\alpha}$ be a lift of $\alpha$ to $\widetilde{S}$. Then the following elements form a complete set of coset representatives for the commutator subgroup of $\Mod(\widetilde{S},\sigma)$:
   \[(T_{\widetilde{\alpha}})^i\bigl(\widetilde{T_\alpha T_{\alpha'}^{-1}}\bigr)^j,\quad 0\le i,j\le p-1,\]
   where $\widetilde{T_\alpha T_{\alpha'}^{-1}}$ denotes a lift of the bounding pair map $T_\alpha T_{\alpha'}^{-1}\in\Mod(S,[\beta])$ to $\Mod(\widetilde{S},\sigma)$,
   except in the case where $S=S_g$ and $g\not\equiv1\pmod p$, in which case the representatives are
   \[(T_{\widetilde{\alpha}})^i\sigma^j,\quad 0\le i,j\le p-1.\]
\end{corollary}

We note that explicit finite generating sets for $\Mod(S_g,[\beta])$ and $\Mod(\widetilde{S},\sigma)$ were given in \cite[Theorem~2]{Dey}, although our arguments do not rely on this description. Combining their generating sets with the coset representatives provided in Corollary \ref{cor: coset beta} and Corollary \ref{cor: coset sigma}, and applying the Reidemeister--Schreier method \cite[Theorem~2.9]{Reidemeister}, one can obtain explicit finite generating sets for the commutator subgroups of $\Mod(S_g,[\beta])$ and $\Mod(\widetilde{S},\sigma)$.

We also note that $\Mod(\widetilde{S},\sigma)$ contains the level-$p$ mapping class group
\[\Mod(S,p)=\text{Ker}(\Mod(S)\to \Aut(H_1(S;\ZZ/p\ZZ))),\]
whose abelianization was computed independently by Sato \cite{SatoAbelianization}, Perron \cite{PerronAbel}, and Putman \cite{Picard}, but this abelianization is not used in our proofs.

The results of this paper are also used in the author's subsequent work on the rigidity of holomorphic maps between the corresponding moduli spaces, where the author also classifies symplectic representations of $\Mod(S,[\beta])$ and $\Mod(\widetilde{S},\sigma)$ of dimension at most $2g$.

\noindent\textbf{Outline.}
Section~2 computes several homology groups of certain subgroups of symplectic groups. Section~3 proves Theorem~\ref{thm: Abel Mod beta} which determines the abelianization of $\Mod(S,[\beta])$, and constructs an explicit abelianization map using Morita's crossed homomorphism. Section~4 proves Theorem~\ref{thm: Abel Mod Sigma} by combining Theorem~\ref{thm: Abel Mod beta} with information about the image of the Prym representation.

\noindent\textbf{Acknowledgements.} I am very grateful to Andrew Putman for many useful comments and constant support.
I am also grateful to Ursula Hamenst\"adt for helpful suggestions.
I would also like to thank Sihao Ma for a helpful conversation.
I am grateful to
the Max Planck Institute for Mathematics in Bonn for its hospitality.

\section{Homology of stabilizers in symplectic groups}

The homological computations in this paper rely on several standard tools from group homology, including the five-term exact sequence \cite[Proposition~VII.6.4]{GroupCohomology}, methods for computing boundary maps, and the long exact sequence in group homology associated to a short exact sequence of modules \cite[Proposition~III.6.1]{GroupCohomology}. We will use these tools without further review.

In this section, we compute several homology groups of certain subgroups of the symplectic group, which will be used in the proofs of the main theorems in the following sections.

We first recall some facts on the homology of the symplectic group.\begin{enumerate} \item We have (see e.g.\ \cite[Theorem 5.1]{Picard})
\begin{equation}\label{eq: H1 Sp 0}
   H_1(\Sp_{2g}(\ZZ);\ZZ)=0, \quad\text{for }g\ge 3,
\end{equation}
and 
\begin{equation}
    H_2(\Sp_{2g}(\ZZ);\ZZ)\cong \ZZ, \quad \text{for }g\ge 4.
\end{equation}
\item By Stein \cite[Theorem 2.13 and Proposition 3.3.a]{Stein}, for any odd $p$
\begin{equation}\label{eq: H2 Sp Z/p 0}
    H_2(\Sp_{2g}(\ZZ/p\ZZ);\ZZ)=0,\quad g\ge 3.
\end{equation}
\item Let $V$ be a vector space over a field $\mathbb{K}$ equipped with a symplectic form. By \cite[Theorem~2.3]{Pollatsek}, we have
\begin{equation}\label{eq: H1 Sp(V) V}
H^1(\Sp(V);V)=0,\qquad \text{if } \operatorname{char}(\mathbb{K})\neq 2,
\end{equation}
while $H^1(\Sp(V);V)\cong \mathbb{K}$
if $\operatorname{char}(\mathbb{K})=2$, by \cite[Theorem~3.3]{Pollatsek}.
\item The principal $p$-congruence subgroup of $\Sp_{2g}(\ZZ)$ is
    \[\Sp_{2g}(\ZZ,p)=\text{Ker}(\Sp_{2g}(\ZZ)\to \Sp_{2g}(\ZZ/p\ZZ)),\]
    whose abelianization is (see Putman \cite[Theorem 1.2]{AndyAbelianization}, Sato \cite[Prop.~2.1]{SatoAbelianization}, or Perron \cite[Prop.~5]{PerronAbel})
    \begin{equation}\label{eq: H1 principal lie alg}
        H_1(\Sp_{2g}(\ZZ,p);\ZZ)\cong \mathfrak{sp}_{2g}(\ZZ/p\ZZ),\quad \text{for }g\ge 3\text{ and }p\text{ odd},
    \end{equation} 
    where $\mathfrak{sp}_{2g}(\ZZ/p\ZZ)\cong \Sym^2((\ZZ/p\ZZ)^{2g})$ is the Lie algebra of $\Sp_{2g}(\ZZ/p\ZZ)$ generated by
    \begin{equation}\label{eq: gens of sp}
       \begin{aligned}
      &A_{ij}=a_i\otimes a_j+a_j\otimes a_i, &&\quad 1\le i\le j\le g; \\
      &B_{ij}=b_i\otimes b_j+b_j\otimes b_i, && \quad 1\le i\le j\le g;\\
      & C_{ij}=a_i\otimes b_j+b_j\otimes a_i, &&\quad 1\le i,j\le g.
  \end{aligned}
  \end{equation}
  where $a_1,b_1,\cdots,a_g,b_g$ is a symplectic basis of $(\ZZ/p\ZZ)^{2g}$. Moreover, the abelianization map sends $X\in \Sp_{2g}(\ZZ,p)$ to $\frac1p(X-I)\in \mathfrak{sp}_{2g}(\ZZ/p\ZZ)$.
\end{enumerate}

Given a nonzero homology class $[\beta]\in H_1(S_g;\ZZ/p\ZZ)$, let
\[\Sp_{2g}(\ZZ/p\ZZ)^{[\beta]} \coloneq \text{Stab}_{\Sp_{2g}(\ZZ/p\ZZ)}([\beta]),\]
we first obtain the abelianization of this group as follows.
\begin{proposition}
    \label{prop: abelSp(2g,Z/pZ)beta}
    Let $g\ge 4$, and $p\ge 3$ be a prime. Then
    $H_1(\Sp_{2g}(\ZZ/p\ZZ)^{[\beta]};\ZZ)=0$.
\end{proposition}
\begin{proof}
The idea is to decompose $\Sp_{2g}(\ZZ/p\ZZ)^{[\beta]}$ into a short exact sequence and then analyze the associated five-term exact sequence. 

  Let $V=(\ZZ/p\ZZ)^{2g}$ equipped with a symplectic form $\hat{i}(-,-)$. Define \[[\beta]^\perp=\{v\in V|\hat{i}(v,[\beta])=0\}.\] Then every element of  $\Sp_{2g}(\ZZ/p\ZZ)^{[\beta]}$ preserves $[\beta]^\perp$ and acts symplectically on the quotient $[\beta]^\perp /[\beta]$, inducing a homomorphism \[\Sp_{2g}(\ZZ/p\ZZ)^{[\beta]}\to \Sp_{2g-2}(\ZZ/p\ZZ).\] This map is surjective, since any symplectic automorphism of $[\beta]^\perp/[\beta]$ can be extended to a symplectic automorphism of $V$ acting trivially on a complementary subspace. Let $K$ denote the kernel of this map. We then obtain the following short exact sequence:
  \begin{equation}\label{ses: kernelK}
      1\to K\to \Sp_{2g}(\ZZ/p\ZZ)^{[\beta]}\to \Sp_{2g-2}(\ZZ/p\ZZ)\to 1.
  \end{equation}
  The structure of the kernel $K$ is as follows. Let 
  \[a_1,b_1=[\beta],a_2,b_2,\cdots,a_g, b_g\]
  be a symplectic basis of $V$ satisfying $\hat{i}(a_j,b_k)=\delta_j^k$, and $ \hat{i}(a_j,a_k)=\hat{i}(b_j,b_k)=0$. Observe that any $f\in K$ satisfies
  \[f(b_1)=b_1,\quad  f(a_j)=a_j+m_j\cdot b_1,\quad f(b_j)=b_j+n_j\cdot b_1 \quad (2\le j\le g),\]
  for some $m_j,n_j\in \ZZ/p\ZZ$. Since $f$ preserves the symplectic form, one then obtains
  \[f(a_1)=a_1+m_1\cdot b_1-\sum\limits_{j=2}^g n_j\cdot a_j+\sum\limits_{j=2}^g m_j\cdot b_j.\]
  Thus $K$ can be identified with the set  
  \[(m_1,m_2,n_2,\cdots,m_g,n_g)\in \ZZ/p\ZZ\rtimes(\ZZ/p\ZZ)^{2g-2}\] with multiplication given by 
  \begin{align*}
      &(m_1,m_2,n_2,\cdots,m_g,n_g)\cdot(m'_1,m'_2,n'_2,\cdots,m'_g,n'_g) \\
      =&(m_1+m_1'+\sum_{i=2}^g(-m_in_i'+n_im_i'),m_2+m_2',n_2+n_2',\cdots,n_g+n_g').
  \end{align*}
  Hence $K$ fits into the following central extension:
  \begin{equation}\label{ses: middleK}
  1\to \ZZ/p\ZZ \to K \to (\ZZ/p\ZZ)^{2g-2} \to 1,
  \end{equation}
  where the map $\ZZ/p\ZZ \to K$ sends $m_1$ to $(m_1,0,\cdots,0)$. This short exact sequence does not split, and we study its five term exact sequence
  \begin{equation}\label{les: K}
      H_2(K)\to H_2((\ZZ/p\ZZ)^{2g-2};\ZZ)\to H_1(\ZZ/p\ZZ;\ZZ)_{(\ZZ/p\ZZ)^{2g-2}} \to H_1(K;\ZZ)\to H_1((\ZZ/p\ZZ)^{2g-2};\ZZ)\to 0.
  \end{equation}
  We next compute the transgression map
  \begin{equation}\label{eq: differential 1}
      d^2:H_2((\ZZ/p\ZZ)^{2g-2};\ZZ)\to H_1(\ZZ/p\ZZ;\ZZ)_{(\ZZ/p\ZZ)^{2g-2}}\cong \ZZ/p\ZZ.
  \end{equation}
  By the Hopf formula, the group $H_2((\ZZ/p\ZZ)^{2g-2};\ZZ)$ is generated by commutator classes, and the transgression sends such a class to
  \begin{equation}\label{eq: descrip of d2}
      d^2([v_1,v_2])=[\widetilde{v_1},\widetilde{v_2}],\quad v_1,v_2\in (\ZZ/p\ZZ)^{2g-2},\quad \widetilde{v_i}\in K\text{ is a lift of }v_i.
  \end{equation} Let
  \[v_1=(m_2,n_2,\cdots,m_g,n_g), \quad v_2=(m_2',n_2',\cdots,m_g',n_g').\]
  Through computations we have
  \[[\widetilde{v_1},\widetilde{v_2}]=2\sum_{i=2}^g(-m_in_i'+n_im_i').\] Since $2$ is a unit in $\ZZ/p\ZZ$ for $p\ge 3$ prime, this shows that the transgression map \eqref{eq: differential 1} is surjective. Therefore, the five-term exact sequence \eqref{les: K} implies that
  \[H_1(K;\ZZ)\cong H_1((\ZZ/p\ZZ)^{2g-2};\ZZ)\cong (\ZZ/p\ZZ)^{2g-2}.\]
  
  We now return to the short exact sequence \eqref{ses: kernelK}, which gives rise to the following five-term exact sequence:
  \begin{equation}\label{les: Sp(2g,Z/p)beta}
  \begin{aligned}
     & H_2(\Sp_{2g}(\ZZ/p\ZZ)^{[\beta]};\ZZ)\to H_2(\Sp_{2g-2}(\ZZ/p\ZZ);\ZZ) \to H_1(K;\ZZ)_{\Sp_{2g-2}(\ZZ/p\ZZ)} \to \\
      \to &H_1(\Sp_{2g}(\ZZ/p\ZZ)^{[\beta]};\ZZ)\to H_1(\Sp_{2g-2}(\ZZ/p\ZZ);\ZZ)\to 0.
  \end{aligned} 
  \end{equation}
  By \eqref{eq: H1 Sp 0}, for $g\ge 4$ we have
  \[H_1(\Sp_{2g-2}(\ZZ/p\ZZ);\ZZ)=0.\]
  Moreover, from the above computation, we obtain
  \[H_1(K;\ZZ)_{\Sp_{2g-2}(\ZZ/p\ZZ)}\cong \big((\ZZ/p\ZZ)^{2g-2}\big)_{\Sp_{2g-2}(\ZZ/p\ZZ)}=0,\]
  since the action of $\Sp_{2g-2}(\ZZ/p\ZZ)$ on $H_1(K;\ZZ)\cong (\ZZ/p\ZZ)^{2g-2}$ is the standard symplectic action. 
  Therefore, analyzing the five-term exact sequence \eqref{les: Sp(2g,Z/p)beta}, we conclude that
  \[H_1(\Sp_{2g}(\ZZ/p\ZZ)^{[\beta]};\ZZ)=0, \quad \text{for } g\ge 4. \qedhere\]
\end{proof}

Next, we compute the second integral homology of $\Sp_{2g}(\ZZ/p\ZZ)^{[\beta]}$.
\begin{proposition}\label{prop: H_2ofSp(2g,Z/p)beta}
    Let $g\ge 4$, and $p\ge 3$ be a prime. Then $H_2(\Sp_{2g}(\ZZ/p\ZZ)^{[\beta]};\ZZ)=0$.
\end{proposition}
\begin{proof}
    The five-term exact sequence associated to the decomposition \eqref{ses: kernelK} is not sufficient to determine $H_2$. Therefore, we study the Hochschild–Serre spectral sequence associated to \eqref{ses: kernelK} in more detail, which is given by 
    \[E^2_{p,q}\eqref{ses: kernelK}=H_p(\Sp_{2g-2}(\ZZ/p\ZZ);H_q(K;\ZZ))\Rightarrow H_{p+q}(\Sp_{2g}(\ZZ/p\ZZ)^{[\beta]};\ZZ).\]
    To show that $H_2(\Sp_{2g}(\ZZ/p\ZZ)^{[\beta]};\ZZ)$ vanishes, it suffices to show that $E_{p,q}^2\eqref{ses: kernelK}=0$ for any $p+q=2$. First, by \eqref{eq: H2 Sp Z/p 0}, we have 
    \[E_{2,0}^2\eqref{ses: kernelK}=H_2(\Sp_{2g-2}(\ZZ/p\ZZ);\ZZ)=0,\quad g\ge 4.\]
    In the proof of Proposition \ref{prop: abelSp(2g,Z/pZ)beta}, we showed that $H_1(K;\ZZ)\cong (\ZZ/p\ZZ)^{2g-2}$. Therefore,
    \[E_{1,1}^2\eqref{ses: kernelK}=H_1(\Sp_{2g-2}(\ZZ/p\ZZ);H_1(K;\ZZ))=H_1(\Sp_{2g-2}(\ZZ/p\ZZ);(\ZZ/p\ZZ)^{2g-2}).\]
    By \cite[Proposition~VI.7.1]{GroupCohomology}, we have
    \[H_1(\Sp_{2g-2}(\ZZ/p\ZZ);(\ZZ/p\ZZ)^{2g-2})\cong H^1(\Sp_{2g-2}(\ZZ/p\ZZ);(\ZZ/p\ZZ)^{2g-2}),\]
    which vanishes by \eqref{eq: H1 Sp(V) V}. It remains to show that  \[E_{0,2}^2\eqref{ses: kernelK}=H_0(\Sp_{2g-2}(\ZZ/p\ZZ);H_2(K;\ZZ))=H_2(K;\ZZ)_{\Sp_{2g-2}(\ZZ/p\ZZ)}\] also vanishes. To do this, we return to the decomposition \eqref{ses: middleK} of $K$ and consider its associated Hochschild–Serre spectral sequence:
    \[E_{p,q}^2\eqref{ses: middleK}=H_p((\ZZ/p\ZZ)^{2g-2};H_q(\ZZ/p\ZZ;\ZZ))\Rightarrow H_{p+q}(K;\ZZ).\]
    To show that $H_2(K;\ZZ)_{\Sp_{2g-2}(\ZZ/p\ZZ)}=0$, we analyze $E_{p,q}^2\eqref{ses: middleK}_{\Sp_{2g-2}(\ZZ/p\ZZ)}$ for $p+q=2$. First we have
    \[E_{0,2}^2\eqref{ses: middleK}=H_0((\ZZ/p\ZZ)^{2g-2};H_2(\ZZ/p\ZZ;\ZZ))=0\] since $H_2(\ZZ/p\ZZ;\ZZ)=0$. Next, since $\ZZ/p\ZZ$ is central in $K$, we have
    \[E_{1,1}^2\eqref{ses: middleK}_{\Sp_{2g-2}(\ZZ/p\ZZ)}=H_1((\ZZ/p\ZZ)^{2g-2};H_1(\ZZ/p\ZZ;\ZZ))_{\Sp_{2g-2}(\ZZ/p\ZZ)}=((\ZZ/p\ZZ)^{2g-2})_{\Sp_{2g-2}(\ZZ/p\ZZ)}=0.\] 
    Since $E_{1,1}^{\infty}\eqref{ses: middleK}$ is a quotient of $E_{1,1}^2\eqref{ses: middleK}$ and taking coinvariants is a right exact functor, we obtain
    \[E_{1,1}^{\infty}\eqref{ses: middleK}_{\Sp_{2g-2}(\ZZ/p\ZZ)}=0.\]
    Hence
    \[H_2(K;\ZZ)_{\Sp_{2g-2}(\ZZ/p\ZZ)}=(E_{2,0}^{\infty}\eqref{ses: middleK})_{\Sp_{2g-2}(\ZZ/p\ZZ)},\]
    Here $E_{2,0}^{\infty}=\text{Ker}(d^2)$ where \[d^2:E_{2,0}\eqref{ses: middleK}=H_2(\Sp_{2g-2}(\ZZ/p\ZZ);\ZZ)\to E_{0,1}\eqref{ses: middleK}\cong \ZZ/p\ZZ.\]
    is the transgression map computed in \eqref{eq: descrip of d2}. We have shown that this map is surjective. Therefore, by the representation theory $\ker(d^2)$ is a nontrivial irreducible $\Sp_{2g-2}(\ZZ/p\ZZ)$-representation. Hence,
    \[H_2(K;\ZZ)_{\Sp_{2g-2}(\ZZ/p\ZZ)}=\text{Ker}(d_2)_{\Sp_{2g-2}(\ZZ/p\ZZ)}=0,\] which implies $H_2(\Sp_{2g}(\ZZ/p\ZZ)^{[\beta]};\ZZ)=0$.
\end{proof}
Next, we study the group $\Sp_{2g}(\ZZ)^{[\beta]}$, which is the preimage of $\Sp_{2g}(\ZZ/p\ZZ)^{[\beta]}$ under the reduction map
\[\Sp_{2g}(\ZZ)\twoheadrightarrow\Sp_{2g}(\ZZ/p\ZZ).\]
We compute its abelianization as follows.
\begin{proposition}\label{prop: abelSp(2g,Z)beta}
    Let $g\ge 4$, and $p\ge 3$ be a prime. Then
    $H_1(\Sp_{2g}(\ZZ)^{[\beta]};\ZZ)\cong \ZZ/p\ZZ$.
\end{proposition}
\begin{proof}
    The reduction map induces the short exact sequence
    \begin{equation}\label{ses: Sp mod p}
        1\to \Sp_{2g}(\ZZ,p) \to \Sp_{2g}(\ZZ)^{[\beta]} \to \Sp_{2g}(\ZZ/p\ZZ)^{[\beta]} \to 1.
    \end{equation}
    The associated five-term exact sequence is
    \begin{equation}\label{les: kernelSp(2g,p)}
  \begin{aligned}
     & H_2(\Sp_{2g}(\ZZ)^{[\beta]};\ZZ)\to H_2(\Sp_{2g}(\ZZ/p\ZZ)^{[\beta]};\ZZ) \to H_1(\Sp_{2g}(\ZZ,p);\ZZ)_{\Sp_{2g}(\ZZ/p\ZZ)^{[\beta]}} \to \\
      \to &H_1(\Sp_{2g}(\ZZ)^{[\beta]};\ZZ)\to H_1(\Sp_{2g}(\ZZ/p\ZZ)^{[\beta]};\ZZ)\to 0.
  \end{aligned} 
  \end{equation}
  By Proposition \ref{prop: abelSp(2g,Z/pZ)beta}, for $g\ge 4$ we have \[H_1(\Sp_{2g}(\ZZ/p\ZZ)^{[\beta]};\ZZ)=0.\]
  Next, we claim that 
  \[H_1(\Sp_{2g}(\ZZ,p);\ZZ)_{\Sp_{2g}(\ZZ/p\ZZ)^{[\beta]}}\cong \ZZ/p\ZZ.\]
  By \eqref{eq: H1 principal lie alg}, for $g\ge 3$ we have
  \[H_1(\Sp_{2g}(\ZZ,p);\ZZ)\cong \mathfrak{sp}_{2g}(\ZZ/p\ZZ),\]
  with a basis given in \eqref{eq: gens of sp}. We compute its
  $\Sp_{2g}(\ZZ/p\ZZ)^{[\beta]}$-coinvariants directly with respect to this basis.
\begin{enumerate}
      \item Take $F\in\Sp_{2g}(\ZZ/p\ZZ)^{[\beta]}$ such that $F(a_i)=a_i+b_1,F(a_1)=a_1+b_i$ ($i\ge 2$) and $F$ fixes all other basis vectors, then
      \begin{align*}
         & FA_{1j}-A_{1j}=C_{ji},\quad FA_{ij}-A_{ij}=C_{j1},\quad i\neq j\ge 2; \\
         & FC_{ii}-C_{ii}=B_{1i},\quad FC_{1k}-C_{1k}=B_{ik}, \quad k\ge 2;\\
         & FA_{ii}-A_{ii}=2C_{i1}+B_{11},\quad FA_{11}-A_{11}=2C_{1i}+B_{ii};\\
         & FA_{1i}-A_{1i}=B_{1i}+C_{ii}+C_{11}.
      \end{align*}
      \item Take $F\in\Sp_{2g}(\ZZ/p\ZZ)^{[\beta]}$ such that $F(b_i)=b_i+b_1,F(a_1)=a_1-a_i$ ($i\ge 2$) and $F$ fixes all other basis vectors, then
      \begin{align*}
          &FA_{1k}-A_{1k}=-A_{ik}, \quad  FA_{11}-A_{11}=A_{ii}-2A_{1i},\quad k\ge 2;\\
          & FC_{1i}-C_{1i}=C_{11}-C_{ii}-C_{i1}.
      \end{align*}
  \end{enumerate}
  From the above, we observe that in the $\Sp_{2g}(\ZZ/p\ZZ)^{[\beta]}$-coinvariants, each $A_{ij}$ ($1\le i\le j\le g$) vanishes, except possibly $A_{11}$; each $B_{ij}$ ($1\le i\le j\le g$) vanishes; each $C_{ij}$ ($1\le i,j \le g)$ vanishes except possibly $C_{ii}$. Moreover, from the last equation in (1) and the last equation in (2), we obtain
  \[[C_{11}]=-[C_{ii}]=-[C_{11}],\] which implies $[C_{ii}]=0$ since $p$ is odd. Therefore, only $A_{11}$ can possibly survive. 
  
  Note that any element $F$ in $\Sp_{2g}(\ZZ/p\ZZ)^{[\beta]}$ fixes $b_1=[\beta]$. Consequently, the term $A_{11}=2 a_1\otimes a_1$ does not appear in any expression of the form $F\cdot x-x$, for $x\in \mathfrak{sp}_{2g}(\ZZ/p\ZZ)$. Therefore 
  \[H_1(\Sp_{2g}(\ZZ,p);\ZZ)_{\Sp_{2g}(\ZZ/p\ZZ)^{[\beta]}} \cong \ZZ/p\ZZ,\]generated by $A_{11}$. 
  Then the five-term exact sequence \eqref{les: kernelSp(2g,p)} implies that $H_1(\Sp_{2g}(\ZZ)^{[\beta]};\ZZ)$ is a quotient of $\ZZ/p\ZZ$. It remains to show that it admits a nontrivial quotient of order $p$. We construct such a quotient as follows:
  \[\chi:\Sp_{2g}(\ZZ)^{[\beta]} \to \ZZ/p\ZZ, \quad F\mapsto \frac1p\hat{i}(F(b),b) \pmod p,\]
  where $b\in \ZZ^{2g}$ is a lift of $[\beta]\in (\ZZ/p\ZZ)^{2g}$. It is straightforward to check that this definition is independent of the choice of the lift $b_1\in \ZZ^{2g}$. We next verify that this map is a homomorphism. For $F,G$ in $\Sp_{2g}(\ZZ)^{[\beta]}$, we have
  \[G(b)=b+p\cdot x, \quad x\in \ZZ^{2g},\]
  and 
  \[\frac1p \hat{i}(F\circ G(b),b)=\frac1p \hat{i}(F(b+p\cdot x),b)=\frac1p \hat{i}(F(b),b)+\hat{i}(F(x),b)\]
  Since $F$ preserves $[b]=[\beta]\in (\ZZ/p\ZZ)^{2g}$, we have 
  \[\hat{i}(F(x),b)=\hat{i}(F(x),F(b))=\hat{i}(x,b)=\hat{i}(\frac1p (G(b)-b),b)=\frac1p \hat{i}( G(b),b) \pmod p.\]
  Hence, we have 
  \[\frac1p \hat{i}(F\circ G(b),b)=\frac1p \hat{i}(F(b),b)+\frac1p \hat{i}(G(b),b).\]
  From definition, we have $\chi(T_{a_1}^p)=1$, so $\chi$ is nontrivial. We conclude that 
  \[H_1(\Sp_{2g}(\ZZ)^{[\beta]};\ZZ)\cong \ZZ/p\ZZ. \qedhere\]
\end{proof}
For later use, we compute the following twisted homology group:
\begin{lemma}\label{lem: twisted H1}
    Let $g\ge 4$ and $p\ge 3$ be prime. Then \[H_1(\Sp_{2g}(\ZZ/p\ZZ)^{[\beta]};\mathfrak{sp}_{2g}(\ZZ/p\ZZ))=0,\]
    where $\Sp_{2g}(\ZZ/p\ZZ)^{[\beta]}$ acts on the Lie algebra $\mathfrak{sp}_{2g}(\ZZ/p\ZZ)$ by conjugation.
\end{lemma}
\begin{proof}
    Recall the short exact sequence \eqref{ses: kernelK}:
    \[1\to K\to \Sp_{2g}(\ZZ/p\ZZ)^{[\beta]}\to \Sp_{2g-2}(\ZZ/p\ZZ)\to 1.\]
    The associated five-term exact sequence with coefficients in $\mathfrak{sp}_{2g}(\ZZ/p\ZZ)$ is given by
    \begin{equation}\label{les: twisted 5 term of kernel K}
  \begin{aligned}
     & H_2(\Sp_{2g}(\ZZ/p\ZZ)^{[\beta]};\mathfrak{sp}_{2g}(\ZZ/p\ZZ))\to H_2(\Sp_{2g-2}(\ZZ/p\ZZ);\mathfrak{sp}_{2g}(\ZZ/p\ZZ)_{K}) \to \\ \to& H_1(K;\mathfrak{sp}_{2g}(\ZZ/p\ZZ))_{\Sp_{2g-2}(\ZZ/p\ZZ)} \to 
     H_1(\Sp_{2g}(\ZZ/p\ZZ)^{[\beta]};\mathfrak{sp}_{2g}(\ZZ/p\ZZ))\to \\ \to & H_1(\Sp_{2g-2}(\ZZ/p\ZZ);\mathfrak{sp}_{2g}(\ZZ/p\ZZ)_{K})\to 0.
  \end{aligned} 
  \end{equation}
  Recall that in the proof of Proposition \ref{prop: abelSp(2g,Z)beta}, we showed that \[\mathfrak{sp}_{2g}(\ZZ/p\ZZ)_{\Sp_{2g}(\ZZ/p\ZZ)^{[\beta]}}\cong\ZZ/p\ZZ.\]
  The same computation gives
  \[\mathfrak{sp}_{2g}(\ZZ/p\ZZ)_{K}\cong \mathfrak{sp}_{2g}(\ZZ/p\ZZ)_{\Sp_{2g}(\ZZ/p\ZZ)^{[\beta]}}\cong\ZZ/p\ZZ.\]
  Moreover, the induced action of $\Sp_{2g-2}(\ZZ/p\ZZ)$ on this coinvariants is trivial. Therefore, we have
  \[H_1(\Sp_{2g-2}(\ZZ/p\ZZ);\mathfrak{sp}_{2g}(\ZZ/p\ZZ)_{K})=H_1(\Sp_{2g-2}(\ZZ/p\ZZ);\ZZ/p\ZZ)\cong H_1(\Sp_{2g-2}(\ZZ/p\ZZ);\ZZ)\otimes \ZZ/p\ZZ=0,\]
  by \eqref{eq: H1 Sp 0} and
  \[H_2(\Sp_{2g-2}(\ZZ/p\ZZ);\mathfrak{sp}_{2g}(\ZZ/p\ZZ)_{K})=H_2(\Sp_{2g-2}(\ZZ/p\ZZ);\ZZ/p\ZZ)\cong H_2(\Sp_{2g-2}(\ZZ/p\ZZ);\ZZ)\otimes \ZZ/p\ZZ=0,\]
  by \eqref{eq: H2 Sp Z/p 0}. The five-term exact sequence \eqref{les: twisted 5 term of kernel K} then reduces to 
  \[H_1(K;\mathfrak{sp}_{2g}(\ZZ/p\ZZ))_{\Sp_{2g-2}(\ZZ/p\ZZ)} \xrightarrow{\cong} 
     H_1(\Sp_{2g}(\ZZ/p\ZZ)^{[\beta]};\mathfrak{sp}_{2g}(\ZZ/p\ZZ)).\]
  Thus, the statement of this lemma reduces to showing that
  \[H_1(K;\mathfrak{sp}_{2g}(\ZZ/p\ZZ))_{\Sp_{2g-2}(\ZZ/p\ZZ)}=0.\]
  Recall the short exact sequence \eqref{ses: middleK}:
  \[1\to \ZZ/p\ZZ \to K \to (\ZZ/p\ZZ)^{2g-2} \to 1.\]
  The associated five-term exact sequence with coefficients in $\mathfrak{sp}_{2g}(\ZZ/p\ZZ)$ is 
  \begin{equation}\label{les: twisted 5 term of mid K}
  \begin{aligned}
     & H_2(K;\mathfrak{sp}_{2g}(\ZZ/p\ZZ))\to H_2((\ZZ/p\ZZ)^{2g-2};\mathfrak{sp}_{2g}(\ZZ/p\ZZ)_{\ZZ/p\ZZ}) \to  H_1(\ZZ/p\ZZ;\mathfrak{sp}_{2g}(\ZZ/p\ZZ))_{(\ZZ/p\ZZ)^{2g-2}} \to 
     \\ \to&H_1(K;\mathfrak{sp}_{2g}(\ZZ/p\ZZ))\to  H_1((\ZZ/p\ZZ)^{2g-2};\mathfrak{sp}_{2g}(\ZZ/p\ZZ)_{\ZZ/p\ZZ})\to 0.
  \end{aligned} 
  \end{equation}
  First, we have
  \[H_1(\ZZ/p\ZZ;\mathfrak{sp}_{2g}(\ZZ/p\ZZ))\cong \frac{\text{Ker}(t-1:\mathfrak{sp}_{2g}(\ZZ/p\ZZ)\to \mathfrak{sp}_{2g}(\ZZ/p\ZZ))}{\operatorname{Im}(\sum\limits_{i=0}^{p-1}t^i:\mathfrak{sp}_{2g}(\ZZ/p\ZZ)\to \mathfrak{sp}_{2g}(\ZZ/p\ZZ))},\]
  where $t$ is the generator of $\ZZ/p\ZZ$ acting on $\mathfrak{sp}_{2g}(\ZZ/p\ZZ)$ by changing the symplectic basis 
  \[a_1,b_1,\cdots,a_g,b_g\longmapsto a_1+b_1,b_1,\cdots,a_g,b_g.\]
  We compute this homology group using the generators in
\eqref{eq: gens of sp} of $\mathfrak{sp}_{2g}(\ZZ/p\ZZ)$ as follows.
  \begin{enumerate}
      \item A direct computation gives
      \begin{equation}\label{eq: Z/p invariants of sp}
      \text{Ker}(t-1)=\mathfrak{sp}_{2g}(\ZZ/p\ZZ)^{\ZZ/p\ZZ}=(\ZZ/p\ZZ)\oplus (\ZZ/p\ZZ)^{2g-2}\oplus \mathfrak{sp}_{2g-2}(\ZZ/p\ZZ)
  \end{equation} 
  where the $(\ZZ/p\ZZ)$-summand is generated by $B_{11}$ and the $(\ZZ/p\ZZ)^{2g-2}$-summand is generated by $B_{1i}$ and $C_{i1}$ for $2\le i \le g$. 
     \item  We compute the image of the norm map. The image of each basis element of $\mathfrak{sp}_{2g}(\ZZ/p\ZZ)$ under $\sum\limits_{i=0}^{p-1}t^i$ is $0$, except possibly for $A_{11}$. In this case, we have
      \[(\sum\limits_{i=0}^{p-1}t^i) A_{11}=pA_{11}+2\sum\limits_{i=1}^{p-1}iC_{11}+\sum\limits_{i=1}^{p-1}i^2 B_{11}=\frac{p(p-1)(2p-1)}{6}B_{11},\]
      which is $2B_{11}$ if $p=3$, and is $0$ if $p>3$.
  \end{enumerate}
  Therefore, we have 
  \[H_1(\ZZ/p\ZZ;\mathfrak{sp}_{2g}(\ZZ/p\ZZ))\cong(\ZZ/p\ZZ)^{2g-2}\oplus \mathfrak{sp}_{2g-2}(\ZZ/p\ZZ), \quad \text{if }p=3, \]
  and 
  \[H_1(\ZZ/p\ZZ;\mathfrak{sp}_{2g}(\ZZ/p\ZZ))\cong(\ZZ/p\ZZ)\oplus(\ZZ/p\ZZ)^{2g-2}\oplus \mathfrak{sp}_{2g-2}(\ZZ/p\ZZ), \quad \text{if }p>3. \]
  The action of the quotient group $(\ZZ/p\ZZ)^{2g-2}$ of $K$ on this homology group is via a lift to $K$ and the action of $K$ on $\mathfrak{sp}_{2g}(\ZZ/p\ZZ)$ by conjugation. That is, 
  \[(m_2,n_2,\cdots,m_g,n_g)\in (\ZZ/p\ZZ)^{2g-2}\]
  acts on $\mathfrak{sp}_{2g}(\ZZ/p\ZZ)$ by changing the symplectic basis
  \begin{align*}
      a_1&\mapsto a_1+m_1\cdot b_1-\sum\limits_{j=2}^g n_j\cdot a_j+\sum\limits_{j=2}^g m_j\cdot b_j, \\
      b_1&\mapsto b_1 ,\\
      a_j&\mapsto m_j\cdot b_1,\quad b_j \mapsto b_j+n_j\cdot b_1 \quad (2\le j\le g).
  \end{align*}
  A direct computation then shows that
  \[H_1(\ZZ/p\ZZ;\mathfrak{sp}_{2g}(\ZZ/p\ZZ))_{(\ZZ/p\ZZ)^{2g-2}}\cong \mathfrak{sp}_{2g-2}(\ZZ/p\ZZ).\]
  Since 
  \[\mathfrak{sp}_{2g-2}(\ZZ/p\ZZ)_{\Sp_{2g-2}(\ZZ)}=0\]
  and the coinvariants functor is right exact, taking
$\Sp_{2g-2}(\ZZ)$-coinvariants of the last three terms of the five-term exact sequence
\eqref{les: twisted 5 term of mid K} yields
  \[H_1(K;\mathfrak{sp}_{2g}(\ZZ/p\ZZ))_{\Sp_{2g-2}(\ZZ/p\ZZ)}\cong H_1((\ZZ/p\ZZ)^{2g-2};\mathfrak{sp}_{2g}(\ZZ/p\ZZ)_{\ZZ/p\ZZ})_{\Sp_{2g-2}(\ZZ/p\ZZ)}.\]
  We now compute the right hand side and show that it vanishes.
 \begin{enumerate}
     \item Since the generator of $\ZZ/p\ZZ$ acts on $\mathfrak{sp}_{2g}(\ZZ/p\ZZ)$ by sending the symplectic basis element $a_1$ to $a_1+b_1$ and fixing all other basis elements, a direct computation gives the coinvariants
     \begin{equation*}
      \mathfrak{sp}_{2g}(\ZZ/p\ZZ)_{\ZZ/p\ZZ}=(\ZZ/p\ZZ)\oplus (\ZZ/p\ZZ)^{2g-2}\oplus \mathfrak{sp}_{2g-2}(\ZZ/p\ZZ),
    \end{equation*}
  where the $(\ZZ/p\ZZ)$-summand is generated by $A_{11}$ and the $(\ZZ/p\ZZ)^{2g-2}$-summand is generated by $A_{1i}$ and $C_{1i}$ for $2\le i \le g$. For simplicity, write $V=(\ZZ/p\ZZ)^{2g-2}$, then
  \begin{equation}\label{eq: Z/p coinvariants of sp}
      \mathfrak{sp}_{2g}(\ZZ/p\ZZ)_{\ZZ/p\ZZ}=(\ZZ/p\ZZ) \oplus V \oplus \Sym^2(V).
  \end{equation}
     \item Note that the quotient group $(\ZZ/p\ZZ)^{2g-2}$ of $K$ acts on \eqref{eq: Z/p coinvariants of sp} via
     \begin{equation}\label{eq: action}
      g\cdot (a,b,c)=(a,b+2ag,c+g\otimes b+b\otimes g+2ag\otimes g),
  \end{equation}
  for $g\in (\ZZ/p\ZZ)^{2g-2}=V$ (we slightly abuse notation by using V both as a group and as a module). Then consider the following short exact sequence of $V$-modules
   \[0\to \Sym^2(V) \to (\ZZ/p\ZZ)\oplus V \oplus \Sym^2(V) \to (\ZZ/p\ZZ)\oplus V\to 0,\]
  which induces the corresponding long exact sequence in homology:
  \begin{equation}\label{les: kernel sym2}
      \begin{aligned}
          \cdots\to & H_1(V;\Sym^2(V))\to H_1(V;\mathfrak{sp}_{2g}(\ZZ/p\ZZ)_{\ZZ/p\ZZ})\to H_1(V;(\ZZ/p\ZZ)\oplus V)\to \\
          \to & H_0(V;\Sym^2(V))\to H_0(V;\mathfrak{sp}_{2g}(\ZZ/p\ZZ)_{\ZZ/p\ZZ})\to H_0(V;(\ZZ/p\ZZ)\oplus V)\to 0.
      \end{aligned}
  \end{equation}
  Via the action \eqref{eq: action}, we obtain the following coinvariants:  
  \begin{align*}
     & H_0(V;\Sym^2(V))=\Sym^2(V)_{V}=\Sym^2(V), \\
     & H_0(V;\mathfrak{sp}_{2g}(\ZZ/p\ZZ)_{\ZZ/p\ZZ})=(\mathfrak{sp}_{2g}(\ZZ/p\ZZ)_{\ZZ/p\ZZ})_{V}=\ZZ/p\ZZ, \\
     & H_0(V;(\ZZ/p\ZZ)\oplus V)=\big((\ZZ/p\ZZ)\oplus V\big)_{V}=\ZZ/p\ZZ.
  \end{align*} 
  Then the long exact sequence \eqref{les: kernel sym2} reduces to
  \begin{equation}\label{les: reduced}
          H_1(V;\ZZ)\otimes \Sym^2(V) \to H_1(V;\mathfrak{sp}_{2g}(\ZZ/p\ZZ)_{\ZZ/p\ZZ}) \to H_1(V;(\ZZ/p\ZZ)\oplus V) \to \Sym^2(V) \to 0.
  \end{equation}
  \item Next, we compute $H_1(V;(\ZZ/p\ZZ)\oplus V)$. Consider the short exact sequence of $V$-modules
  \begin{equation}\label{ses: V Z/p}
      0\to V \to (\ZZ/p\ZZ)\oplus V \to \ZZ/p\ZZ\to 0,
  \end{equation}
  where $g\in V$ acts on $(\ZZ/p\ZZ)\oplus V$ via
  \[g\cdot (a,b)=(a,b+2ag).\]
  The associated long exact sequence in homology is
  \begin{equation*}
      \begin{aligned}
          \cdots  \to& H_2(V;\ZZ/p\ZZ)\to  H_1(V;V)\to H_1(V;(\ZZ/p\ZZ)\oplus V)\to H_1(V;\ZZ/p\ZZ)\to \\
          \to & H_0(V;V)\to H_0(V;(\ZZ/p\ZZ)\oplus V)\to H_0(V;\ZZ/p\ZZ)\to 0, 
      \end{aligned}
  \end{equation*}
  which, after identifying $H_0(V;-)$ as $V$-coinvariants, becomes
  \begin{equation*}
      \begin{aligned}
          \cdots  \to& H_2(V;\ZZ/p\ZZ)\to  H_1(V;V)\to H_1(V;(\ZZ/p\ZZ)\oplus V)\to V\to \\
          \to & V\to \ZZ/p\ZZ\to \ZZ/p\ZZ\to 0.
      \end{aligned}
  \end{equation*}
  This implies that
  \[H_1(V;(\ZZ/p\ZZ)\oplus V)\cong \text{coker}(\partial: H_2(V;\ZZ/p\ZZ)\to  H_1(V;V)).\]
  Here we have
    \[H_2(V;\ZZ/p\ZZ)\cong \wedge^2 V \oplus V,\]
    which follows from the universal coefficient theorem:
  \[0\to H_2(V;\ZZ)\otimes (\ZZ/p\ZZ)\to H_2(V;\ZZ/p\ZZ)\to \text{Tor}_1(H_1(V;\ZZ),\ZZ/p\ZZ)\to 0.\]
  We now compute the connecting homomorphism $\partial$ using the standard (right) bar resolution. Recall that $\partial$ is obtained by lifting a cycle in $C_2(V;\ZZ/p\ZZ)$ to $C_2(V;(\ZZ/p\ZZ)\oplus V)$ and then applying the chain differential.
  
  The $\wedge^2 V$-summand of $H_2(V;\ZZ/p\ZZ)$ is generated by the classes
  \[([g|h]-[h|g])\otimes_{\ZZ[V]} 1, \quad g,h\in V,1\in \ZZ/p\ZZ,\] where $1\in \ZZ/p\ZZ$ lifts to $(1,0)\in (\ZZ/p\ZZ)\oplus V$. Applying the bar differential gives
  \begin{align*}
  &d\big(([g|h]-[h|g])\otimes_{\ZZ[V]} (1,0)\big)\\
  =&\big(g[h]\otimes_{\ZZ[V]}(1,0)-[gh]\otimes_{\ZZ[V]}(1,0)+[g]\otimes_{\ZZ[V]}(1,0)\big)\\
      &-\big(h[g]\otimes_{\ZZ[V]}(1,0)-[hg]\otimes_{\ZZ[V]}(1,0)+[h]\otimes_{\ZZ[V]}(1,0)\big) \\
      =& [h]\otimes_{\ZZ[V]}\big(g^{-1}\cdot(1,0)-(1,0)\big)-[g]\otimes_{\ZZ[V]}\big(h^{-1}\cdot(1,0)-(1,0)\big) \\
      =& [h]\otimes_{\ZZ[V]}(0,-2g)-[g]\otimes_{\ZZ[V]}(0,-2h).
  \end{align*}
  Thus
  \[\partial(([g|h]-[h|g])\otimes_{\ZZ[V]} 1)=2([g]\otimes h-[h]\otimes g).\]
  Identifying $H_1(V,V)\cong V\otimes V$ and noting that $2$ is a unit in $\ZZ/p\ZZ$, we conclude that the image of the $\wedge^2 V$-summand under $\partial$ is precisely $\wedge^2 V\subset V\otimes V$. 
  
  The $V$-summand of $H_2(V;\ZZ/p\ZZ)$ arises from the $\text{Tor}_1$-term in the universal coefficient sequence, which is generated by cycles
  \[z_g=\sum_{i=0}^{p-1}[g|g^i]\otimes_{\ZZ[V]} 1
\in C_2(V;\ZZ/p\ZZ),\quad g\in V.\]
  Lifting $1\in \ZZ/p\ZZ$ to $(1,0)\in (\ZZ/p\ZZ)\oplus V$ and applying the bar differential, we obtain
\begin{align*}
d(z_g)&=\sum_{i=0}^{p-1}d([g|g^i]\otimes_{\ZZ[V]} (1,0))\\
&=\sum_{i=0}^{p-1}(g[g^i]\otimes_{\ZZ[V]} (1,0)-[g^{i+1}]\otimes_{\ZZ[V]} (1,0)+[g]\otimes_{\ZZ[V]} (1,0)) \\
&=\sum_{i=0}^{p-1}[g^i]\otimes_{\ZZ[V]}\big(g^{-1}\cdot(1,0)-(1,0)\big)+p[g]\otimes_{\ZZ[V]} (1,0)\\
&=\sum_{i=0}^{p-1}[g^i]\otimes_{\ZZ[V]}(0,-2g).
\end{align*}
This is zero in $H_1(V;V)\cong V\otimes V$, and hence $\partial$ vanishes on the $V$-summand. 

Therefore the image of $\partial$ is $\wedge^2 V$, and
  \[H_1(V;(\ZZ/p\ZZ)\oplus V)\cong \text{coker}(\partial)\cong (V\otimes V)/\wedge^2 V\cong\Sym^2(V).\]
  \item We now return to the exact sequence \eqref{les: reduced}, which reduces to 
  \[V\otimes \Sym^2(V) \to H_1(V;\mathfrak{sp}_{2g}(\ZZ/p\ZZ)_{\ZZ/p\ZZ}) \to 0.\]
  Applying the right-exact functor of $\Sp_{2g-2}(\ZZ/p\ZZ)$-coinvariants yields
  \[(V\otimes \Sym^2(V))_{\Sp_{2g-2}(\ZZ/p\ZZ)}
   \longrightarrow
   H_1(V;\mathfrak{sp}_{2g}(\ZZ/p\ZZ)_{\ZZ/p\ZZ})_{\Sp_{2g-2}(\ZZ/p\ZZ)}
  \longrightarrow 0.\]
Since
\[
(V\otimes \Sym^2(V))_{\Sp_{2g-2}(\ZZ/p\ZZ)}=0
\]
by a direct computation, it follows that
\[
H_1(V;\mathfrak{sp}_{2g}(\ZZ/p\ZZ){\ZZ/p\ZZ})_{\Sp{2g-2}(\ZZ/p\ZZ)}=0,
\]
as claimed. \qedhere
 \end{enumerate} 
\end{proof}

This lemma has the following important consequence:
\begin{corollary}\label{cor: surjectivity of Level P to Beta}
    The map induced by group inclusion $\Sp_{2g}(\ZZ,p)\hookrightarrow \Sp_{2g}(\ZZ)^{[\beta]}$ 
    \[H_2(\Sp_{2g}(\ZZ,p);\ZZ)\to H_2(\Sp_{2g}(\ZZ)^{[\beta]};\ZZ)\]
    is surjective.
\end{corollary}
\begin{proof}
    Recall the short exact sequence \eqref{ses: Sp mod p}
    \[ 1\to \Sp_{2g}(\ZZ,p) \to \Sp_{2g}(\ZZ)^{[\beta]} \to \Sp_{2g}(\ZZ/p\ZZ)^{[\beta]} \to 1,\]
    whose associated Hochshild-Serre spectral sequence is
    \[E_{p,q}^2\eqref{ses: Sp mod p}=H_p(\Sp_{2g}(\ZZ/p\ZZ)^{[\beta]};H_q(\Sp_{2g}(\ZZ,p);\ZZ))\Rightarrow H_{p+q}(\Sp_{2g}(\ZZ)^{[\beta]};\ZZ).\]
    The image of the map in the statement \[H_2(\Sp_{2g}(\ZZ,p);\ZZ)\to H_2(\Sp_{2g}(\ZZ)^{[\beta]};\ZZ)\] is precisely $E_{0,2}^{\infty}\eqref{ses: Sp mod p}$. Hence, to prove that this map is surjective, it suffices to show that $E_{1,1}^{\infty}\eqref{ses: Sp mod p}$ and $E_{2,0}^{\infty}\eqref{ses: Sp mod p}$ vanish. The term
    \[E_{1,1}^{2}\eqref{ses: Sp mod p}=H_1(\Sp_{2g}(\ZZ/p\ZZ)^{[\beta]};H_1(\Sp_{2g}(\ZZ,p);\ZZ))=H_1(\Sp_{2g}(\ZZ/p\ZZ)^{[\beta]};\mathfrak{sp}_{2g}(\ZZ/p\ZZ))\]
    vanishes by Lemma \ref{lem: twisted H1}, and the term
    \[E_{2,0}^{2}\eqref{ses: Sp mod p}=H_2(\Sp_{2g}(\ZZ/p\ZZ)^{[\beta]};H_0(\Sp_{2g}(\ZZ,p);\ZZ))=H_2(\Sp_{2g}(\ZZ/p\ZZ)^{[\beta]};\ZZ)\]
    vanishes by Proposition \ref{prop: H_2ofSp(2g,Z/p)beta}. This completes the proof.
\end{proof}

\section{The abelianization of $\Mod(S,[\beta])$}
In this section, we prove Theorem~\ref{thm: Abel Mod beta}, which determines the abelianization of
\[
\Mod(S,[\beta])=\text{Stab}_{\Mod(S)}([\beta]), \quad [\beta]\in H_1(S;\ZZ/p\ZZ)^*.
\]
The action of $\Mod(S)$ on $H_1(S;\ZZ)$ induces a surjective homomorphism
\[
\Mod(S)\longrightarrow \Sp_{2g}(\ZZ),
\]
whose kernel is the Torelli subgroup, denoted by $\mathcal{I}(S)$. Restricting this action to $\Mod(S,[\beta])$ yields the short exact sequence
\begin{equation}\label{ses: Torelli kernel}
1\to \mathcal{I}(S)\to \Mod(S,[\beta])\to \Sp_{2g}(\ZZ)^{[\beta]}\to 1.
\end{equation}
The homology groups of $\Sp_{2g}(\ZZ)^{[\beta]}$ needed for our computation were determined in the previous section. The abelianization of the Torelli subgroup was computed by Johnson \cite[Theorem~1,~Theorem~4]{JohnsonTorelliIII} for $g\ge 3$. More precisely, if $S=S_{g,1}$ or $S_g^1$, then
\begin{equation}\label{eq: Abel Torelli nonclosed}
    H_1(\mathcal{I}(S);\ZZ)\cong \wedge^3 H\oplus (2\text{-torsion subgroup}),
\end{equation}
where $H=H_1(S;\ZZ)$. If $S=S_g$, then
\begin{equation}\label{eq: Abel Torelli closed}
    H_1(\mathcal{I}(S);\ZZ)\cong (\wedge^3 H/H)\oplus (2\text{-torsion subgroup}),
\end{equation}
where $H$ is embedded in $\wedge^3 H$ via
\[
H\longrightarrow \wedge^3 H,\qquad
h\longmapsto h\wedge\left(\sum_{i=1}^g a_i\wedge b_i\right),
\]
where $a_1,b_1,\cdots,a_g,b_g$ is a symplectic basis of $H$ satisfying
\[\hat{i}(a_j,b_k)=\delta_j^k,\quad \hat{i}(a_j,a_k)=\hat{i}(b_j,b_k)=0.\]
We assume that $[b_1]=[\beta]$ from now on. Moreover, the isomorphisms \eqref{eq: Abel Torelli nonclosed} and
\eqref{eq: Abel Torelli closed} are $\Sp_{2g}(\ZZ)$-equivariant, where $H$ is endowed with the standard $\Sp_{2g}(\ZZ)$-action.

We begin by determining the image of the Torelli group in the abelianization of
$\Mod(S,[\beta])$.
\begin{proposition}\label{prop: image of H1 Torelli}
    Let $g\ge 3$. The image of the map
    \[H_1(\mathcal{I}(S);\ZZ)\to H_1(\Mod(S,[\beta]);\ZZ)\]
    induced by the group inclusion $\mathcal{I}(S)\to \Mod(S,[\beta])$ is isomorphic to the coinvariants
    \[(\wedge^3H)_{\Sp_{2g}(\ZZ)^{[\beta]}}=(\wedge^3H)/\langle Fa\wedge Fb\wedge Fc - a\wedge b\wedge c :a,b,c \in H,F\in\Sp_{2g}(\ZZ)^{[\beta]} \rangle,\]
    if $S=S_{g,1}$ or $S_g^1$. If $S=S_g$, then the image is isomorphic to
    \[(\wedge^3H/H)_{\Sp_{2g}(\ZZ)^{[\beta]}}.\]
\end{proposition}
\begin{proof}
    First, we show that the $2$-torsion subgroup of $H_1(\mathcal{I}(S);\ZZ)$ maps trivially to $H_1(\Mod(S,[\beta]);\ZZ)$. By definition, $\Mod(S,[\beta])$ contains the level-$p$ mapping class group
   \[\Mod(S,p)=\text{Ker}(\Mod(S)\to \Aut(H_1(S;\ZZ/p\ZZ)).\]
    Therefore, the homomorphism
    \begin{equation}\label{eq: H1 Torelli to H1}
        H_1(\mathcal{I}(S);\ZZ)\to H_1(\Mod(S,[\beta]);\ZZ)
    \end{equation}
    factors as
    \[H_1(\mathcal{I}(S);\ZZ)\to H_1(\Mod(S,p),\ZZ)\to H_1(\Mod(S,[\beta]);\ZZ).\]
    It is known (see \cite[Lemma 11.7]{SatoAbelianization}, \cite[Lemma 4.3]{AndyAbelianization}) that the image of 
    \[H_1(\mathcal{I}(S);\ZZ)\to H_1(\Mod(S,p),\ZZ)\]
    consists of $p$-torsion elements. Hence the same holds for the image of \eqref{eq: H1 Torelli to H1}. 
    
    Since $\Sp_{2g}(\ZZ)^{[\beta]}$ acts trivially on $H_1(\Mod(S,[\beta]);\ZZ)$, the map \eqref{eq: H1 Torelli to H1} factors through 
    \[H_1(\mathcal{I}(S);\ZZ)_{\Sp_{2g}(\ZZ)^{[\beta]}}\to H_1(\Mod(S,[\beta]);\ZZ).\]
    It remains to show that this map is injective after quotienting out the $2$-torsion subgroup of $H_1(\mathcal{I}(S);\ZZ)$. The proof strategy used in \cite{SatoAbelianization} and \cite{AndyAbelianization} for the abelianization of the level-$p$ mapping class group does not apply in our setting. Indeed, their arguments rely on the nontriviality of a wrong-way map, called the relative Johnson homomorphism introduced in \cite[Theorem~5.8]{BFP}. In our case, however,
    the relative Johnson homomorphism associated to $\Mod(S,[\beta])$ is trivial.
    
    Instead, we prove the injectivity by comparing spectral sequences and performing a diagram chase, together with an essential use of Corollary \ref{cor: surjectivity of Level P to Beta}. Consider the following commutative diagram:
 \begin{equation}\label{eq: map btw two short exact sequences}
    \xymatrix{1 \ar[r] & \mathcal{I}(S) \ar[r] \ar[d] & \Mod(S,p) \ar[r] \ar[d] & \Sp_{2g}(\ZZ,p) \ar[r] \ar[d] & 1 \\
    1 \ar[r] & \mathcal{I}(S) \ar[r] & \Mod(S,[\beta]) \ar[r] & \Sp_{2g}(\ZZ)^{[\beta]} \ar[r] & 1.}
\end{equation}
This induces a morphism between the Hochschild-Serre spectral sequences associated to the first and second rows. In particular, it induces a commutative diagram between the corresponding five-term exact sequences.
\begin{equation}\label{eq: map btw five term exact sequence}
    \xymatrix{
    \cdots \ar[r] &H_2(\Sp_{2g}(\ZZ,p);\ZZ) \ar[r] \ar[d] & H_1(\mathcal{I}(S);\ZZ)_{\Sp_{2g}(\ZZ,p)} \ar[r] \ar[d] & H_1(\Mod(S,p);\ZZ) \ar[r] \ar[d] & \cdots \\
    \cdots \ar[r] &H_2(\Sp_{2g}(\ZZ)^{[\beta]};\ZZ) \ar[r] & H_1(\mathcal{I}(S);\ZZ)_{\Sp_{2g}(\ZZ)^{[\beta]}} \ar[r] & H_1(\Mod(S,[\beta]);\ZZ) \ar[r] &\cdots.}
\end{equation}
The map in the first row of \eqref{eq: map btw five term exact sequence}
\[H_2(\Sp_{2g}(\ZZ,p);\ZZ) \to H_1(\mathcal{I}(S);\ZZ)_{\Sp_{2g}(\ZZ,p)}\]is trivial after quotienting out the \(2\)-torsion subgroup of $H_1(\mathcal{I}_g;\ZZ)$ (see
\cite[Lemma~11.10]{SatoAbelianization} or \cite[Theorem~4.2]{AndyAbelianization}).
Moreover, Corollary \ref{cor: surjectivity of Level P to Beta} implies that the first vertical map in \eqref{eq: map btw five term exact sequence},
\[H_2(\Sp_{2g}(\ZZ,p);\ZZ)\to H_2(\Sp_{2g}(\ZZ)^{[\beta]};\ZZ)\]
is surjective. Therefore, by diagram chasing, the map in the second row of
\eqref{eq: map btw five term exact sequence},
\[H_2(\Sp_{2g}(\ZZ)^{[\beta]};\ZZ) \to H_1(\mathcal{I}(S);\ZZ)_{\Sp_{2g}(\ZZ)^{[\beta]}}\]
is also trivial after quotienting out the $2$-torsion subgroup of $H_1(\mathcal{I}(S);\ZZ)$. This completes the proof by the exactness of the five term exact sequence.
\end{proof}

We then compute these $\Sp_{2g}(\ZZ)^{[\beta]}$-coinvariants as follows.
\begin{proposition}\label{prop: coinvariants}
    Let $g\ge 3$ and let $p$ be an odd prime. Then 
    \[(\wedge^3H)_{\Sp_{2g}(\ZZ)^{[\beta]}}\cong \ZZ/p\ZZ\]
    generated by the class of $a_1\wedge a_2 
    \wedge b_2$. Moreover,
    \[(\wedge^3H/H)_{\Sp_{2g}(\ZZ)^{[\beta]}}\cong\begin{cases}
         \ZZ/p\ZZ & \text{ if } g \equiv 1 \pmod p; \\
        0  & \text{ if } g \not\equiv 1 \pmod p.
    \end{cases}\]
\end{proposition}
\begin{proof}
    We first compute $(\wedge^3H)_{\Sp_{2g}(\ZZ)^{[\beta]}}$. Since $\wedge^3H$ is generated by $x\wedge y \wedge z$ for $x,y,z$ distinct elements in the symplectic basis $\{a_1,b_1,\cdots,a_g,b_g\}$, We successively simplify these generators in the coinvariants as follows.
    \begin{enumerate}
    \item Let $x,y,z$ be distinct elements in $\{a_2,b_2,\cdots,a_g,b_g\}$. Let $\widehat{x}$ be the symplectic basis vector such that $\widehat{i}(x,\widehat{x})=\pm 1$. Let $F\in \Sp_{2g}(\ZZ)^{[\beta]}$ be the transvection about $\widehat{x}$, then 
    \[F\cdot(x\wedge y \wedge z)-x\wedge y \wedge z=\pm \widehat{x}\wedge y \wedge z.\]
    \item Let $y,z$ be distinct elements in $\{a_2,b_2,\cdots,a_g,b_g\}$. Let $F\in \Sp_{2g}(\ZZ)^{[\beta]}$ be the transvection about $b_1$, then 
    \[F\cdot(a_1\wedge y \wedge z)-a_1\wedge y \wedge z=-b_1\wedge y \wedge z.\]    
    \item For any $x\in\{a_2,b_2,\cdots,a_g,b_g\}$, let $\widehat{x}$ be the symplectic basis vector such that $\widehat{i}(x,\widehat{x})=\pm 1$. Let $F\in \Sp_{2g}(\ZZ)^{[\beta]}$ be the transvection about $x$, then 
    \[F\cdot(a_1\wedge b_1 \wedge \widehat{x})-a_1\wedge b_1 \wedge \widehat{x}=\pm a_1 \wedge b_1 \wedge x.\]
    \item Let $y,z$ be distinct elements in $\{a_2,b_2,\cdots,a_g,b_g\}$. Let $F\in \Sp_{2g}(\ZZ)^{[\beta]}$ be the $p$-th power of the transvection about $a_1$, then 
    \[F\cdot(b_1\wedge y \wedge z)-b_1\wedge y \wedge z=p a_1\wedge y \wedge z.\]
    Furthermore, if $\hat{i}(y,z)=0$, let $F\in \Sp_{2g}(\ZZ)^{[\beta]}$ be the the transvection about $z$, and let $\widehat{z}$ be the symplectic basis vector such that $\widehat{i}(z,\widehat{z})=\pm 1$, then
    \[F\cdot (a_1\wedge y\wedge\widehat{z})-a_1\wedge y\wedge\widehat{z}=\pm a_1\wedge y\wedge z.\]
    If $\hat{i}(y,z)=1$, may as well assume that $y=a_i$ and $z=b_i$ for some $i\ge 2$. Then let $F\in \Sp_{2g}(\ZZ)^{[\beta]}$ be the factor swap that exchanges $a_i,b_i$ and $a_j,b_j$ ($i\neq j\ge 2$), then
    \[F\cdot (a_1\wedge a_i\wedge b_i)-a_1\wedge a_i\wedge b_i=a_1\wedge a_j\wedge b_j-a_1\wedge a_i\wedge b_i.\]
    \end{enumerate}
    This implies that $(\wedge^3 H)_{\Sp_{2g}(\ZZ)^{[\beta]}}$ is a quotient of $\ZZ/p\ZZ$ generated by the class \[[a_1\wedge a_2\wedge b_2]=[a_1\wedge a_i\wedge b_i],\quad \text{for all }2\le i \le g.\] Moreover, we construct the following homomorphism of abelian groups
    \begin{equation}\label{eq: map phi}
        \begin{aligned}
        \phi:\wedge^3 H &\to \ZZ/p\ZZ \\
         x \wedge y \wedge z &\mapsto \hat{i}([\beta],x)\cdot \hat{i}(y,z)+\hat{i}([\beta],y)\cdot \hat{i}(z,x)+\hat{i}([\beta],z)\cdot \hat{i}(x,y) \pmod p.
    \end{aligned}
    \end{equation}
    This map is surjective since
    \[\phi(a_1\wedge a_2\wedge b_2)=\hat{i}([\beta],a_1)\cdot \hat{i}(a_2,b_2)+\hat{i}([\beta],a_2)\cdot \hat{i}(b_2,a_1)+\hat{i}([\beta],b_2)\cdot \hat{i}(a_1,a_2)=-1 \pmod p.\]
    Note that the homomorphism $\phi$ is invariant under the action of $\Sp_{2g}(\ZZ)^{[\beta]}$, since every $F\in \Sp_{2g}(\ZZ)^{[\beta]}$ fixes the class $[\beta]$ and preserves the symplectic pairing $\widehat{i}$. Therefore, $\phi$ induces a surjective homomorphism
    \[(\wedge^3 H)_{\Sp_{2g}(\ZZ)^{[\beta]}}\twoheadrightarrow\ZZ/p\ZZ.\]
    It follows that
    \[(\wedge^3 H)_{\Sp_{2g}(\ZZ)^{[\beta]}}\cong\ZZ/p\ZZ.\] 
    
    To determine the coninvariants $(\wedge^3H/H)_{\Sp_{2g}(\ZZ)^{[\beta]}}$, it suffices to compute the image of $H$ inside $(\wedge^3 H)_{\Sp_{2g}(\ZZ)^{[\beta]}}$. The image of a basis element $h\in H$ is
    \[[h\wedge(\sum_{i=1}^g a_i\wedge b_i)]=\begin{cases}
         (g-1)[a_1\wedge a_2\wedge b_2], & \text{ if } h=a_1; \\
        0,  & \text{ if } h\in \{b_1,a_2,b_2,\cdots,a_g,b_g\}.
    \end{cases}\]
    Hence, if $g\not\equiv 1 \pmod p$, we have
    \[(\wedge^3 H/H)_{\Sp_{2g}(\ZZ)^{[\beta]}}=0.\]
    If $g\equiv 1 \pmod p$, noticing that $H$ lies in the kernel of the surjection \eqref{eq: map phi}, we have
    \[(\wedge^3 H/H)_{\Sp_{2g}(\ZZ)^{[\beta]}}=\ZZ/p\ZZ.\qedhere\]
\end{proof}

\begin{remark}\label{rmk: represent by bpm}
In the abelianization of the Torelli subgroup $\mathcal{I}(S)$, the class $a_1\wedge a_2\wedge b_2\in \wedge^3H$ is represented by the bounding pair map $T_{\alpha}T_{\alpha'}^{-1}\in \mathcal{I}(S)$ where $\alpha$, $\alpha'$ are simple closed curves indicated in Figure \ref{fig: surface with aab}. 

\begin{figure}
    \centering\includegraphics[width=0.6\linewidth]{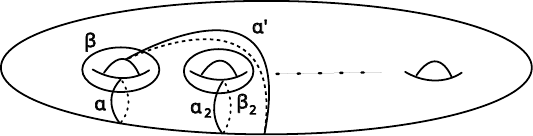}\caption{\label{fig: surface with aab}}\end{figure}
\end{remark}

We now prove Theorem \ref{thm: Abel Mod beta} that determines the abelianization of $\Mod(S,[\beta])$. 
\begin{proof}[Proof of Theorem~\ref{thm: Abel Mod beta}]  
Recall the short exact sequence \eqref{ses: Torelli kernel}
\[   1\to \mathcal{I}(S)\to \Mod(S,[\beta]) \to \Sp_{2g}(\ZZ)^{[\beta]} \to 1.\]
The associated five-term exact sequence is 
\begin{equation}\label{les: kernelTorelli}
  \begin{aligned}
     & H_2(\Mod(S_g,[\beta]);\ZZ)\to H_2(\Sp_{2g}(\ZZ)^{[\beta]};\ZZ) \to H_1(\mathcal{I}(S);\ZZ)_{\Sp_{2g}(\ZZ)^{[\beta]}} \to \\
      \to &H_1(\Mod(S,[\beta]);\ZZ)\to H_1(\Sp_{2g}(\ZZ)^{[\beta]};\ZZ)\to 0.
  \end{aligned} 
  \end{equation}
  By Proposition \ref{prop: abelSp(2g,Z)beta}, we have
  \[H_1(\Sp_{2g}(\ZZ)^{[\beta]};\ZZ)\cong \ZZ/p\ZZ,\quad\text{for }g\ge 4.\] 
  By Proposition \ref{prop: image of H1 Torelli} and Proposition \ref{prop: coinvariants}, the image of the map
  \[H_1(\mathcal{I}(S);\ZZ)_{\Sp_{2g}(\ZZ)^{[\beta]}} 
      \to H_1(\Mod(S,[\beta]);\ZZ)\]
  is isomorphic to $\ZZ/p\ZZ$, except in the case where $S=S_g$ and $g\not\equiv 1 \pmod p$, in which case the image is trivial. Therefore, the five-term exact sequence \eqref{les: kernelTorelli} implies that
  \[H_1(\Mod(S,[\beta]);\ZZ)=\ZZ/p\ZZ,\]
  if $S=S_g$ and $g\not\equiv 1 \pmod p$. In all other cases, it yields a short exact sequence
\[0\to \ZZ/p\ZZ\to H_1(\Mod(S,[\beta]);\ZZ) \to \ZZ/p\ZZ\to 0.\]
It remains to show that this short exact sequence splits. It suffices to construct a surjective homomorphism
\[\Mod(S,[\beta])\longrightarrow \ZZ/p\ZZ\oplus\ZZ/p\ZZ.\]

The first component is the homomorphism arising from the abelianization map of $\Sp_{2g}(\ZZ)^{[\beta]}$ in the proof of Proposition \ref{prop: abelSp(2g,Z)beta}, which gives a surjection
\[\begin{aligned}
\chi:\Mod(S,[\beta])&\longrightarrow \ZZ/p\ZZ,\\
f&\longmapsto \frac{1}{p}\widehat{i}(f\cdot[\beta],[\beta])\pmod p,
\end{aligned}\]
where $\widehat{i}$ denotes the algebraic intersection pairing.

For the second component, we use Morita's crossed homomorphism
\[\kappa:\Mod(S_{g,1})\longrightarrow H_1(S_g;\ZZ)\]
from \cite[Section 6]{MoritaI}. Recall that, for
$\gamma\in\pi_1(S_g)$, viewed as an element of the point-pushing subgroup of
$\Mod(S_{g,1})$, Morita's crossed homomorphism satisfies
\begin{equation}\label{eq: image of point pushing}
    \kappa(\gamma)=(2-2g)[\gamma],
\end{equation}
by \cite[Proposition 3.1]{MoritaII}.

We define
\[\begin{aligned}
\psi:\Mod(S,[\beta])&\longrightarrow \ZZ/p\ZZ,\\
f&\longmapsto \widehat{i}\bigl(\kappa(\widetilde{f}),[\beta]\bigr)\pmod p,
\end{aligned}\]
where $\widetilde f$ is defined as follows:
\[\widetilde f=
\begin{cases}
f, & \text{if } S=S_{g,1},\\[2mm]
\text{image of $f$ under }
\Mod(S_g^1)\xrightarrow{\operatorname{Cap}}\Mod(S_{g,1}),
& \text{if } S=S_g^1,\\[2mm]
\text{any lift of $f$ under }
\Mod(S_{g,1})\xrightarrow{\operatorname{Forget}}\Mod(S_g),
& \text{if } S=S_g,\ g\equiv1\pmod p.
\end{cases}\]
In the last case, $\psi$ is independent of the choice of lifts of $f$. Indeed, any two lifts $\widetilde f_1,\widetilde f_2\in\Mod(S_{g,1})$ of $f$ differ by an element of the point-pushing subgroup:
\[\widetilde{f}_1\widetilde{f}_2^{-1}\in \pi_1(S_g)\subset \Mod(S_{g,1})\]
in the point pushing subgroup. Thus, by \eqref{eq: image of point pushing}, we have
\[\kappa(\widetilde{f}_1\widetilde{f}_2^{-1})=(2-2g)[\gamma],\qquad \text{for some }\gamma\in\pi_1(S_g).\]
Therefore we have
\[\kappa(\widetilde{f}_1)=\kappa(\widetilde{f}_1\widetilde{f}_2^{-1}\widetilde{f}_2)=\kappa(\widetilde{f}_1\widetilde{f}_2^{-1})+\widetilde{f}_1\widetilde{f}_2^{-1}\cdot\kappa(\widetilde{f}_2)=(2-2g)[\gamma]+\kappa(\widetilde{f}_2).\]
Since $g\equiv1\pmod p$, we have $(2-2g)\equiv0\pmod p$. Thus $\kappa(\widetilde f_1)$ and $\kappa(\widetilde f_2)$ have the same intersection pairing with $[\beta]$ modulo $p$.

Since $\kappa$ is a crossed homomorphism, we have
\[\kappa(\widetilde{f}\widetilde{h})
=\kappa(\widetilde{f})+\widetilde{f}\cdot\kappa(\widetilde{h}).\]
It follows that $\psi$ is a homomorphism:
\begin{align*}
\psi(fh)
&=\widehat{i}\bigl(\kappa(\widetilde{f}\widetilde{h}),[\beta]\bigr) \\
&=\widehat{i}\bigl(\kappa(\widetilde{f})
+\widetilde{f}\cdot\kappa(\widetilde{h}),[\beta]\bigr) \\
&=\widehat{i}\bigl(\kappa(\widetilde{f}),[\beta]\bigr)
+\widehat{i}\bigl(\widetilde{f}\cdot \kappa(\widetilde{h}),[\beta]\bigr) \\
&=\widehat{i}\bigl(\kappa(\widetilde{f}),[\beta]\bigr)
+\widehat{i}\bigl(\kappa(\widetilde{h}),\widetilde{f}^{-1}\cdot [\beta]\bigr) \\
&=\psi(f)+\psi(h),
\end{align*}
where the last equality follows from the fact that $f$ fixes
$[\beta]$ by definition.

We now verify that the map
\begin{equation}\label{eq: abel 1}
    (\chi,\psi):\Mod(S,[\beta])
\longrightarrow \ZZ/p\ZZ\oplus\ZZ/p\ZZ
\end{equation}
is surjective. Let $\alpha,\alpha'$ be the curves shown in
Figure \ref{fig: surface with aab}. A direct application of the combinatorial definition of $\kappa$ in \cite[Section~6]{MoritaI} gives
\[\kappa(\widetilde{T_\alpha^p})=p[\alpha],\quad \kappa(\widetilde{T_\alpha T_{\alpha'}^{-1}})=2[\alpha].\]
It follows that
\[(\chi,\psi)(T_\alpha^p)=
\left(\frac{1}{p}\widehat{i}(T_\alpha^p\cdot[\beta],[\beta])\pmod p,\,
\widehat{i}\bigl(\kappa(\widetilde{T_\alpha^p}),[\beta]\bigr)\pmod p
\right)
=(1,0),\]
and
\[(\chi,\psi)(T_\alpha T_{\alpha'}^{-1})=
\left(\frac{1}{p}\widehat{i}\bigl((T_\alpha T_{\alpha'}^{-1})\cdot[\beta],[\beta]\bigr)\pmod p,\,
\widehat{i}\bigl(\kappa(\widetilde{T_\alpha T_{\alpha'}^{-1}}),[\beta]\bigr)\pmod p
\right)
=(0,2).\]
Since $p$ is an odd prime, we obtain that $(\chi,\psi)$ is surjective, which completes the proof.
\end{proof}

\section{The abelianization of $\Mod(\widetilde{S},\sigma)$}
In this section, we prove Theorem \ref{thm: Abel Mod Sigma}, which determines the abelianization of 
\[\Mod(\widetilde{S},\sigma)= \text{C}_{\Mod(\widetilde{S})}(\sigma),\]
where $\sigma$ is a generator of the deck transformation group of the unbranched cyclic $p$-fold cover $\widetilde{S}\to S$. A key ingredient is the image of the Prym representation associated to the cover $\widetilde{S}\to S$, which is the action of $\Mod(\widetilde{S},\sigma)$ on $H_1(\widetilde{S};\ZZ)$. 

Let $G\cong\ZZ/p\ZZ$ be the deck transformation group of $\widetilde{S}\to S$, where $p$ is any positive integer (not necessarily prime). In the case $S=S_g$, by the Chevalley-Weil Theorem 
\[H_1(\widetilde{S};\QQ)\cong \QQ^2\oplus (\QQ[G])^{2g-2}\cong \QQ^{2g}\oplus \bigoplus_{0\neq C\in X(G)} \QQ(\zeta_C)^{2g-2},\]
as a $\QQ[G]$-module, where $X(G)$ denotes the set of cyclic factor groups of $G$ and $\zeta_C$ is a primitive $|C|$-th root of unity. By Schur's lemma, the action of $\Mod(\widetilde{S},\sigma)$ preserves each isotypic component. On the trivial component $\QQ^{2g}$, it is the action of $\Mod(\widetilde{S},\sigma)$ on $H_1(S;\QQ)$, hence its image is $\Sp_{2g}(\ZZ)^{[\beta]}$. 

For each nontrivial factor group $C$ of $G$, the Prym representation restricted to the component $\QQ(\zeta_C)^{2g-2}$ preserves a nondegenerate sesquilinear skew-Hermitian $\QQ(\zeta_C)$-valued form induced by the $\QQ[G]$-valued augmented intersection form 
\begin{align*}
    H_1(\widetilde{S};\QQ)\times H_1(\widetilde{S};\QQ) &\to \QQ[G] \\
    \langle a,b\rangle&\coloneq \sum\limits_{f\in G}\widehat{i}(a,fb)f,
\end{align*}
where $\widehat{i}$ is the algebraic intersection pairing on $H_1(\widetilde{S};\ZZ)$. Let $U^\#_{2g-2}(\ZZ[\zeta_C])$ denote the subgroup of $\Aut(\QQ(\zeta_C)^{2g-2})$ consisting of automorphisms that preserve this
form, have entries in $\ZZ[\zeta_C]$, and have determinant equal to an even power of $\zeta_C$. Looijenga \cite[Theorem 2.4]{PrymRep} proved that $U^\#_{2g-2}(\ZZ[\zeta_C])$ is precisely the image of the Prym representation
on this isotypic component. Consequently, the Prym representation can be written as 
\begin{equation}\label{eq: image of the Prym rep}
    \text{Prym}:\Mod(\widetilde{S},\sigma)\longrightarrow \Sp_{2g}(\ZZ)^{[\beta]}\times \prod_{0\neq C\in X(G)} U^\#_{2g-2}(\ZZ[\zeta_C]).
\end{equation}

We now prove Theorem~\ref{thm: Abel Mod Sigma}.
\begin{proof}[Proof of Theorem \ref{thm: Abel Mod Sigma}]
    If $S=S_{g,1}$ or $S_g^1$, then
    \[\Mod(\widetilde{S},\sigma)\cong \Mod(S,[\beta]),\]
    so its abelianization is isomorphic to $\ZZ/p\ZZ\oplus \ZZ/p\ZZ$ by Theorem~\ref{thm: Abel Mod beta}. It therefore remains to consider the case $S=S_g$, for which we have the central extension
    \begin{equation}\label{eq: central extension}
        1\to \langle\sigma\rangle \to \Mod(\widetilde{S},\sigma) \to \Mod(S,[\beta])\to 1.
    \end{equation}
    First, by forgetting all the punctures of $\widetilde{S_{g,1}}$, we obtain a surjection
    \[\Mod(\widetilde{S_{g,1}},\sigma)\to \Mod(\widetilde{S},\sigma).\]
    This implies that $H_1(\Mod(\widetilde{S},\sigma);\ZZ)$ is a quotient of
    \[H_1(\Mod(\widetilde{S_{g,1}},\sigma);\ZZ)\cong \ZZ/p\ZZ\oplus \ZZ/p\ZZ.\]
    If $g \equiv 1 \pmod p$, since \eqref{eq: central extension} induces a surjection
    \[H_1(\Mod(\widetilde{S},\sigma);\ZZ) \to H_1(\Mod(S,[\beta]);\ZZ)\cong \ZZ/p\ZZ\oplus \ZZ/p\ZZ,\]
    we conclude that
    \[H_1(\Mod(\widetilde{S},\sigma);\ZZ) \cong \ZZ/p\ZZ\oplus \ZZ/p\ZZ.\]
    If $g \not\equiv 1 \pmod p$, in this case
    \[H_1(\Mod(S,[\beta]);\ZZ)\cong H_1(\Sp_{2g}(\ZZ)^{[\beta]};\ZZ)\cong \ZZ/p\ZZ.\]
    Since $p$ is a prime, there is only one nontrivial cyclic factor group of the deck transformation group $\ZZ/p\ZZ$. In this case, the Prym representation \eqref{eq: image of the Prym rep} associated to the cover $\widetilde{S}\to S$ simplifies to
    \begin{equation}\label{eq: Prym rep for prime p}
        \text{Prym}:\Mod(\widetilde{S},\sigma)\longrightarrow \Sp_{2g}(\ZZ)^{[\beta]}\times U^{\#}_{2g-2}(\ZZ[\zeta_p]),
    \end{equation}
    where $\zeta_p$ is a $p$-th root of unity. The determinant map \[\det:U^{\#}_{2g-2}(\ZZ[\zeta_p])\longrightarrow
      \langle \zeta_p\rangle\cong \ZZ/p\ZZ\]
      is surjective since $p$ is odd. Then \eqref{eq: Prym rep for prime p} induces the following homomorphism
      \[\Phi:\Mod(\widetilde{S},\sigma)\xrightarrow{\operatorname{Prym}}
\Sp_{2g}(\mathbb Z)^{[\beta]}
\times U^{\#}_{2g-2}(\mathbb Z[\zeta_p])
\xrightarrow{H_1(-;\mathbb Z)\times\det} H_1(\Sp_{2g}(\ZZ)^{[\beta]};\ZZ)\oplus \langle \zeta_p\rangle\cong \ZZ/p\ZZ\oplus \ZZ/p\ZZ.\]
    We now show that $\Phi$ is surjective. First, let $\alpha$ be a simple closed curve in $S$ representing the basis vector $a_1\in H_1(S;\ZZ)$ and intersecting $\beta$ once. Let $\widetilde{\alpha}$ be the preimage of $\alpha$ under the covering map $\widetilde{S}\to S$, which is a single simple closed curve. By construction, the Dehn twist $T_{\widetilde{\alpha}}$ lies in $\Mod(\widetilde{S},\sigma)$, and its image under the Prym representation \eqref{eq: Prym rep for prime p} is
    \[\text{Prym}(T_{\widetilde{\alpha}})=(T_{a_1}^p,I).\]
    Therefore,
    \[\Phi(T_{\widetilde{\alpha}})=(1,0),\]
    since the class of $T_{a_1}^p$ generates $H_1(\Sp_{2g}(\ZZ))^{[\beta]};\ZZ)$ by Proposition \ref{prop: abelSp(2g,Z)beta}.
    
    On the other hand, the image of the deck transformation $\sigma$ under the Prym representation \eqref{eq: Prym rep for prime p} is
    \[\text{Prym}(\sigma)=\bigl(I,\operatorname{diag}(\zeta_p,\cdots,\zeta_p)\bigr),\]
    and hence
    \[\Phi(\sigma)=(0,\zeta^{2g-2})= (0,1)\in \ZZ/p\ZZ\oplus \ZZ/p\ZZ,\]
    where the last equality follows from the assumption that $g \not\equiv 1 \pmod p$. 
    
    Hence $\Phi$ is surjective, and therefore
    \[H_1(\Mod(\widetilde{S},\sigma);\ZZ)\cong \ZZ/p\ZZ\oplus \ZZ/p\ZZ.\qedhere\]
\end{proof}

\bibliographystyle{alpha}
\bibliography{ref}

\end{document}